\documentclass[12pt]{amsart}
\usepackage{amsmath}
\usepackage{amscd}
\usepackage{amssymb}
\usepackage{pdfsync}
 \usepackage[all,cmtip]{xy}
\usepackage{graphicx}
 \usepackage{placeins}
 \usepackage{float}
 \usepackage{eurosym}
\usepackage{color}

\numberwithin{equation}{section}

\newcommand{\nc}{\newcommand}

\newtheorem{thm}{Theorem}[section]

\newtheorem{lemma}[thm]{Lemma}
\newtheorem{lem}[thm]{Lemma}

\newtheorem{thm-dfn}[thm]{Theorem-Definition}

\theoremstyle{definition}

\newtheorem{rmk}{Remark}[section]

\numberwithin{equation}{section}

\nc{\bA}{{\mathbb A}}
\nc{\bB}{{\mathbb B}}
\nc{\bC}{{\mathbb C}}
\nc{\bD}{{\mathbb D}}
\nc{\bE}{{\mathbb E}}
\nc{\bF}{{\mathbb F}}
\nc{\bG}{{\mathbb G}}
\nc{\bH}{{\mathbb H}}
\nc{\bI}{{\mathbb I}}
\nc{\bJ}{{\mathbb J}}
\nc{\bK}{{\mathbb K}}
\nc{\bL}{{\mathbb L}}
\nc{\bM}{{\mathbb M}}
\nc{\bN}{{\mathbb N}}
\nc{\bO}{{\mathbb O}}
\nc{\bP}{{\mathbb P}}
\nc{\bQ}{{\mathbb Q}}
\nc{\bR}{{\mathbb R}}
\nc{\bS}{{\mathbb S}}
\nc{\bT}{{\mathbb T}}
\nc{\bU}{{\mathbb U}}
\nc{\bV}{{\mathbb V}}
\nc{\bW}{{\mathbb W}}
\nc{\bZ}{{\mathbb Z}}
\nc{\bX}{{\mathbb X}}
\nc{\bY}{{\mathbb Y}}

\newcommand{\cO}{{\mathcal O}}
\newcommand{\cA}{{\mathcal A}}

\newcommand{\cN}{{\mathcal N}}

\newcommand{\cL}{{\mathcal L}}

\newcommand{\Hom}{{\operatorname{Hom}}}

\newcommand{\Sym}{{\operatorname{Sym}}}
\newcommand{\het}{{\operatorname{ht}}}

\newcommand{\Lb}{{\mathfrak{b}}}

\newcommand{\Lg}{{\mathfrak g}}

\newcommand{\Ln}{{\mathfrak{n}}}

\newcommand{\Lh}{{\mathfrak{h}}}

\nc{\al}{{\alpha }}
\nc{\ga}{{\gamma }}
\nc{\de}{{\delta }}
\nc{\del}{{\partial }}
\nc{\ep}{{\varepsilon }}
\nc{\vap}{{\epsilon }}

\nc{\ze}{{\zeta }}
\nc{\et}{{\eta }}

\nc{\vth}{{\vartheta }}

\nc{\io}{{\iota }}
\nc{\ka}{{\kappa }}
\nc{\vrho}{{\varrho}}
\nc{\si}{{\sigma }}
\nc{\ups}{{\upsilon }}
\nc{\vphi}{{\varphi }}
\nc{\om}{{\omega }}

\nc{\Ga}{{\Gamma }}
\nc{\De}{{\Delta }}
\nc{\nab}{{\nabla}}
\nc{\Th}{{\Theta }}
\nc{\Si}{{\Sigma }}
\nc{\Ups}{{\Upsilon }}
\nc{\Om}{{\Omega }}

\newcommand{\fb}{{\mathfrak{b}}}

\newcommand{\fg}{{\mathfrak g}}

\newcommand{\fn}{{\mathfrak{n}}}

\newcommand{\tk}{{\mathbb{C}}}

\newcommand{\la}{\langle}
\newcommand{\ra}{\rangle}

\nc{\ot}{\otimes}

\newcommand{\wt}{\widetilde}

\nc{\oh}{{\operatorname{H}}}
\nc{\gr}{{\operatorname{gr}}}
\nc{\Supp}{{\operatorname{Supp}}}
\nc{\PSupp}{{\operatorname{PSupp}}}

\newcommand{\beqn}{\begin{equation*}}
\newcommand{\eeqn}{\end{equation*}}

\newcommand{\beq}{\begin{equation}}
\newcommand{\eeq}{\end{equation}}

\newcommand{\bern}{\begin{eqnarray*}}
\newcommand{\eern}{\end{eqnarray*}}

\begin{document}

\title{The null-cone and cohomology of vector bundles on flag varieties}

\author{Kari Vilonen}\address{School of Mathematics and Statistics, University of Melbourne, VIC 3010, Australia, also Department of Mathematics and Statistics, University of Helsinki, Helsinki, Finland}
\email{kari.vilonen@unimelb.edu.au, kari.vilonen@helsinki.fi}
\thanks{The first author was supported in part by NSF grants DMS-1402928 \& DMS-1069316, the Academy of Finland, and the ARC grant DP150103525.}

\author{Ting Xue}
\address{School of Mathematics and Statistics, University of Melbourne, VIC 3010, Australia, also Department of Mathematics and Statistics, University of Helsinki, Helsinki, Finland} 
\email{ting.xue@unimelb.edu.au}

\thanks{The second author was supported in part by the ARC grants DP150103525, DE160100975, and the Academy of Finland.}

\begin{abstract}
We study the null-cone of a semi-simple algebraic group acting on a number of copies of its Lie algebra via the diagonal adjoint action. We show that the null cone has rational singularities in the case of $SL_3$. We observe  by example that the null-cone is not normal in general and that the normalization of the null-cone does not have rational singularities in general. This is achieved by computing cohomology of certain vector bundles on flag varieties.
\end{abstract}

\maketitle

\section{Introduction}

In this paper we study the null-cone of a semi-simple algebraic group acting on a number of copies of its Lie algebra via the diagonal adjoint action. Such actions and their generalizations have been considered by various authors, see, for example,  \cite{CM,KW1,KW2,LMP}. Our interest in the question was motivated by applications in the study of ordinary deformations of Galois representations. This point of view is due to Snowden who  studied the case $\Lg=\mathfrak{sl}_2$ in~\cite{S}. In the case of $\Lg=\mathfrak{sl}_2$ he shows that the null-cones (and other related spaces that come up in ordinary deformations of Galois representations) are Cohen-Macaulay but not Gorenstein. In the case of characteristic zero the method of Snowden amounts to proving that the null-cone has rational singularities. In this paper we show that for $\Lg=\mathfrak{sl}_3$ the null-cones do still have rational singularities, and hence are Cohen-Macaulay. However, we also show that this fails in general. For example, it is not difficult to see that the null cone is not normal when the group is of type $B_2$. In  type $A_5$ we further show that the normalization of the null-cone does not have rational singularities. We do this by giving estimates on cohomology groups of equvariant vector bundles on flag varieties. 

In~\cite{H} Hesselink shows that the nilpotent cone of a semi-simple Lie algebra has rational singularities. As the cotangent  bundle of the flag variety is a resolution of singularities of the nilpotent cone, one is  reduced to proving  cohomology vanishing of the symmetric algebra of the tangent bundle of the flag variety. By utilizing the Koszul complex this is reduced to the Borel-Weil-Bott theorem. In our case of multiple copies of the Lie algebra the situation is more complicated as one is reduced to analyzing the cohomology of symmetric powers of direct sums of copies of the tangent bundle of a flag variety. This, in turn, is reduced to analyzing cohomology groups of the whole tensor algebra of the tangent bundle. We carry out this analysis just far enough to produce the counterexamples. The calculation of the cohomology groups of the tensor algebra of the tangent bundle appears to be a difficult problem.

The paper is organized as follows. In  section 2 we introduce the null-cone and in section 3 we make some general remarks on cohomology of equivariant vector bundles on flag varieties. In section 4 we reduce the question of rational singularities to a question of cohomology of equivariant vector bundles on flag varieties. In section 5 we compute large enough piece of the cohomology of certain equivariant vector bundles and in section 6 we state and prove our main results. We manage to perform this computation sufficiently far to produce the counterexamples.

As our results in this paper are mostly  counterexamples we will just work in characteristic zero and will only comment on the finite characteristic case. 

{\bf Acknowledgement.} This work was partially performed at the Max Planck Institute for Mathematics in Bonn, the Research Institute of Mathematical Sciences in Kyoto, and the Mathematical Sciences Research Institute in Berkeley. The authors thank these institutions for their hospitality. We also thank Frank Calegari for bringing this question to our attention. 
Furthermore, we thank the referee for careful reading of the manuscript and for helpful suggestions. 
\section{The Null-cone}

We work over a field of characteristic zero which we can and will, for simplicity, take to be $\bC$. 
Let $\fg$ be a semi-simple Lie algebra, and let us write $G$ for the corresponding adjoint group $\operatorname{Aut}(\fg)^0$. 

We write $\Lb$ for Borel subalgebras of $\Lg$ and $\Ln=[\Lb,\Lb]$ for their nil-radicals. Similarly, we write $B$ for Borel subgroups of $G$.

The group $G$ acts diagonally on $\Lg^{\oplus r}$ via the adjoint action. We write 
\beqn
\cA_r=\{(x_1,\ldots,x_r)\in\Lg^{\oplus r}\,|\,f(x_1,\ldots,x_r)=0\ \forall\ f\in\tk[\Lg^{\oplus r}]^G_+\}
\eeqn
for the (reduced) invariant theory null-cone associated to this action. The null-cone itself is, in general, a non-reduced scheme, but in what follows we will work with its underlying reduced scheme structure.
 It is not difficult to see that this variety can also be described in the following manner (see \cite{KW2}):
\beqn
\cA_r=\{(x_1,\ldots,x_r)\in\Lg^{\oplus r}\,|\,\exists \text{ Borel subalgebra }\Lb\text{ such that }x_i\in\Ln\text{ for all }i\}.
\eeqn
Furthermore, in the  case $\Lg=\mathfrak{sl}_n$, we can view $\cA_r$ as 
\beqn
\cA_r=\{(x_1,\ldots,x_r)\in\Lg^{\oplus r}\,|\,x_{i_1}\cdots x_{i_n}=0 \ \text{for all choices of} \ i_1,\dots,i_n\}\,.
\eeqn

The variety $\cA_r$ has a natural $G$-equivariant resolution of singularities which can be described as follows. Let $X$ denote the flag variety of $G$ then the resolution is given by:
\beqn
\xymatrix{{\wt \cA}_r:=G\times_B\Ln^{\oplus r}\cong(T^*X)^{\oplus r} \ar[r]^-{\ \varphi_r\ }&\cA_r}
\eeqn
where  the map $\varphi_r$ is given by $\varphi_r(g,x_1,\ldots,x_r)=(g\,x_1,\ldots,g\,x_r)$. We write $$\pi:{\wt \cA}_r\cong(T^*X)^{\oplus r} \to X$$ for the projection. 

 In the case of $r=1$ the null-cone $\cA_1=\cN$ is the usual nilpotent cone and the resolution $\varphi_1: T^*X \to \cN$ is often referred to as the Springer resolution. In this case the null-cone is reduced and the ring of invariants exhibits it as a complete intersection.  In the case of $r=2$ Charbonnel and Moreau~\cite{CM} defined a nilpotent bicone and they show that it is a complete intersection and is in general not reduced. The variety $\cA_2$ is an irreducible component of the bicone. 
 
 The case of $\Lg=\mathfrak{sl}_2$  was studied by Snowden~\cite{S} where he proves that in this case $\cA_r$ has rational singularities and so is, in particular, Cohen-Macaulay. 

As $\varphi_r : {\wt \cA}_r \to \cA_r$ is a resolution of singularities of an affine variety we conclude that
\beqn
\text{$\Gamma({\wt \cA}_r , \cO_{{\wt \cA}_r })\ \ $  is the normalization of  $\ \ \bC[\cA_r]$}\,.
\eeqn
It is easy to see that $\cA_r$ has rational singularities, i.e., that $R\varphi_{r*}  \cO_{{\wt \cA}_r } \cong \cO_{\cA_r}$ if and only if
\beqn
\oh^i({\wt \cA}_r,\cO_{{\wt \cA}_r})=0 \ \ \text{for} \ i>0 \qquad \text{and}\qquad \text{$\Gamma({\wt \cA}_r , \cO_{{\wt \cA}_r })=\bC[\cA_r]$}\,.
\eeqn

\section{Cohomology of equivariant vector bundles}

Let us consider the flag variety  $X=G/B$, where we think of having chosen a particular Borel subgroup $B$ as a base point. Then we have an equivalence of categories
\begin{equation*}
\{\text{$G$-equivariant vector bundles on $X$}\} \longleftrightarrow \{\text{$B$-representations}\}\,.
\end{equation*}
Given a $B$-representation $E$ we will use the same symbol $E$ to denote the corresponding vector bundle. Thus, $\oh^k(X,E)$ stands for the cohomology groups of the $G$-equivariant vector bundle associated to $E$. The cohomology groups $\oh^k(X,E)$  are $G$-representations, of course. 

Let us write $T\subset B$ for a maximal torus and $X^*(T)$ for the group of characters $\lambda:T\to\bG_m$. We also write $\Phi\subset X^*(T)$ for the set of roots  and we choose a positive root system $\Phi^+$ such that the roots in $B$ are negative. We also write
\beqn
X^+(T) \ = \ \{\lambda\in X^*(T)\mid \la\lambda,\check\alpha\ra\geq 0 \ \text{for} \ \alpha\in\Phi^+\}
\eeqn
for the dominant weights; here $\check\alpha$ stands for the coroot associated to $\alpha$. We write $L(\lambda)$ for the irreducible representation of $G$ associated to the highest weight $\lambda\in X^+(T)$. Given a representation $V$ of $G$ we write 
\beqn
\Supp(V) = \{\lambda \in  X^+(T)\mid L(\lambda) \ \text{occurs as a direct summand in}\ V\}
\eeqn
for the support of $V$; it is, of course, a subset of the dominant weights $X^+(T)$. 

We can view each $\lambda\in X^*(T)$ also as a character of $B$ and in this manner $\lambda$ gives rise to a $G$-equivariant line bundle $\cL_\lambda$. If $\lambda$ is dominant then $L(\lambda) = \oh^0(X,\cL_\lambda)$. The ``dot" action of the Weyl group $W$ on $X^*(T)$ is given by $w\cdot \lambda= w(\lambda+\rho)-\rho$; here $\rho$, as usual, is half the sum of positive roots. We recall the statement of the Borel-Weil-Bott theorem:
\beq\label{BWB}
\begin{gathered}
\text{If there exists a $w\in W$ such that $w\cdot \lambda\in X^+(T)$ then  }
\\
\oh^k(X,\cL_{\lambda})\simeq\left\{\begin{array}{ll}\oh^0(X,\cL_{w\cdot\lambda})&\text{if }k=l(w)\\0&\text{otherwise}\,.\end{array}\right.
\end{gathered}
\eeq
This statement says, in particular, that if there is no $w\in W$ such that $w\cdot \lambda$ is dominant then the cohomology of $\cL_{\lambda}$ vanishes in all degrees. 

Consider a $B$-representation $E$ and the cohomology groups $\oh^k(X,E)$ of the corresponding $G$-equivariant vector bundle.  We will give a simple upper bound for $\Supp(\oh^k(X,E))$. Let us choose a filtration of the $B$-representation $E$ such that the associated graded $\gr E$ is a direct sum of one dimensional representations, i.e., of characters of $T$. Let us write $\chi(E)$ for the set of characters appearing in this direct sum decomposition:
\beqn
\gr E \ \cong \ \bigoplus_{\lambda_i\in \chi(E)} \lambda_i^{\oplus n_i} \,.
\eeqn
Thus we obtain a filtration of the  $G$-bundle $E$ such that the associated graded is a direct sum of $G$-equivariant line bundles. From this it is easy to conclude:
\beqn
\Supp(\oh^k(X,E)) \subset \{\lambda\in X^+(T)\mid \text{$w\cdot \lambda\in\chi(E) $ for a $w\in W$ with $\ell(w)=k$}\}\,.
\eeqn
Let us call the right-hand side of this equality the {\it potential support} of $\oh^k(X,E)$ and we write $\PSupp(\oh^k(X,E))$ for it.

Let us recall that either by a direct calculation or using the Hodge decomposition for the flag variety and, identifying $\fn$ and $(\fg/\fb)^*$ via the Killing form, we conclude:
\begin{equation}
\label{Hodge}
 \oh^i(X,\wedge^k\fn)=
\left\{\begin{array}{ll} 0 \text{ if $i\neq k$}
\\
\text{trivial $G$-module of dimension $\ |\{w\in W\,|\,l(w)=k\}|$ if $i=k$}.\end{array}\right.  
\end{equation}
Finally, as part of the direct calculation one makes use of the following:
\begin{lem}
\label{distinct roots}
If $\chi$ is a sum of distinct negative roots and $w\cdot \chi$ is dominant then $w\cdot \chi=0$.
\end{lem}
For a proof of this lemma see, for example, \cite[6.18 Proposition]{J}.

\section{Some reductions}

Recall that we have reduced the question of normality and rational singularities to the study of
\beqn
\oh^i({\wt \cA}_r,\cO_{{\wt \cA}_r}) =  \oh^i((T^*X)^{\oplus r}, \cO_{(T^*X)^{\oplus r}} )\,.
\eeqn
We have
\beqn
\oh^i({\wt \cA}_r,\cO_{{\wt \cA}_r}) =  \oh^i((T^*X)^{\oplus r}, \cO_{(T^*X)^{\oplus r}} )=\oh^i(X,\text{Sym}((\Lg/\Lb)^{\oplus r}))\,.
\eeqn
In the latter equality we use the fact that the tangent bundle of $X$ is $TX= \Lg/\Lb$ as $G$-equivariant vector bundles; recall our notational convention that the $B$-representation $\Lg/\Lb$ also stands for the corresponding $G$-equivariant bundle. Thus we are reduced to analyzing the cohomology groups $\oh^i(X,\text{Sym}((\Lg/\Lb)^{\oplus r}))$.
In particular,  $\cA_r$ is normal if and only if
\beqn
\bC[\cA_r] = \oh^0(X,\text{Sym}((\Lg/\Lb)^{\oplus r}))\,.
\eeqn
We then conclude that 
\beqn
\text{$\cA_r$ is normal if and only if the map $\Sym(\Lg^{\oplus r}) \to \oh^0(X,\Sym((\Lg/\Lb)^{\oplus r}))$ is onto\,}
\eeqn
and 
\beq\label{vanishing..}
\begin{gathered}
\text{$\cA_r$ has rational singularities if and only if it is}
\\
\text{normal and $\oh^i(X,\text{Sym}((\Lg/\Lb)^{\oplus r}))=0$ for $i>0$\,.}
\end{gathered}
\eeq
 
We will next make some very general reductions for the vanishing of the higher cohomology in \eqref{vanishing..}. As we will show later, the higher cohomology does not vanish in general and hence the general reductions are not so useful. We will be able to obtain more precise statements later. However, the general remarks below are perhaps helpful as a general guide. 

We begin with some simple lemmas:
\begin{lemma}\label{lem-sts}
Suppose that $E$ is a $B$-module.
Then for all $i\geq 1$,
\beqn
\oh^i(X,\Sym(E^{\oplus r}))=0\text{ for all }r\geq 1\Longleftrightarrow \oh^i(X,E^{\otimes r})=0\text{ for all }r\geq 1.
\eeqn
\end{lemma}
\begin{proof}
We have $\text{Sym}(E^{\oplus r})=(\text{Sym}(E))^{\otimes r}=\bigoplus (S^{i_1}E\otimes\cdots\otimes S^{i_r}E)$. So $E^{\otimes r}$ is a direct summand of $\text{Sym}(E^{\oplus r})$ and thus ``$\Rightarrow$'' follows. In the other direction, assume that $\oh^i(X,E^{\otimes r})=0\text{ for all } r\geq 1.$ Then we have 
$\oh^i(X,(E^{\oplus r})^{\otimes k}) =0$ for all $r,k\geq 1$. Now $\text{Sym}^k(E^{\oplus r})$ is a direct summand of $(E^{\oplus r})^{\otimes k}$. It follows that $\oh^i(X,\Sym^k(E^{\oplus r})) =0$ for all $r,k\geq 1$ and thus $\oh^i(X,\Sym(E^{\oplus r})) =0$ for all $r\geq 1$. \end{proof}

In the same way we obtain
\begin{lemma}\label{lem-stw}
Suppose that $E$ is a $B$-module.
Then 
\beqn
\oh^i(X,\wedge^r(E^{\oplus k}))=0\text{ for all }r,k\geq 1,i\geq r+1\Longleftrightarrow \oh^i(X,E^{\otimes r})=0\text{ for all }r\geq 1, i\geq r+1.
\eeqn
\end{lemma}

Let us now consider the  vanishing statement:
\beq
\label{cohomology vanishing}
\oh^i(X,\text{Sym}((\Lg/\Lb)^{\oplus r}))=0\text{ for }\ i\geq 1\,.
\eeq
Consider the short exact sequence
\beqn
0\to\Lb^{\oplus r}\to\Lg^{\oplus r}\to(\Lg/\Lb)^{\oplus r}\to 0. 
\eeqn
Its associated  Koszul complex provides a resolution of the $\operatorname{Sym}^{m}(\Lg^{\oplus r})$-module $\operatorname{Sym}^{m}((\Lg/\Lb)^{\oplus r})$ as follows:
\beq\label{koszul complex}
\dots \wedge^{i}(\Lb^{\oplus r}) \ot \operatorname{Sym}^{m-i} (\Lg^{\oplus r})\to\dots \to \operatorname{Sym}^{m}(\Lg^{\oplus r})\to \text{Sym}^{m}((\Lg/\Lb)^{\oplus r})\to0.
\eeq
As the $\operatorname{Sym}^{m-i} (\Lg^{\oplus r})$ are $G$-representations the corresponding $G$-equivariant vector bundles are trivial.
We conclude that the vanishing statement~\eqref{cohomology vanishing} is equivalent to 
\beq
\label{b wedge statement}
 \oh^{i+j}(X, \wedge^j(\Lb^{\oplus r}))=0\text{ for all }i\geq 1 \,.
\eeq
This statement holds for any particular $r$. 

On the other hand, using Lemma \ref{lem-sts}, we see that  the vanishing statement~\eqref{cohomology vanishing} for all $r$ is equivalent to 
\beqn
\oh^i(X, (\Lg/\Lb)^{\otimes r})=0\text{ for all }i,r\geq 1.
\eeqn
Similarly, using Lemma~\ref{lem-stw} we see that the vanishing statement~\eqref{b wedge statement} for all $r$ is equivalent to
\beqn
\oh^i(X, \Lb^{\otimes r})=0\text{ for all }r\geq 1,\ i\geq r+1.
\eeqn

The short exact sequence 
\beqn
0\to\Lb \to\Lg \to\Lg/\Lb \to 0
\eeqn
will  give us further information. If we consider it as a two-step complex $\Lb \to\Lg$ with cohomology $\Lg/\Lb$ and pass to the associated  tensor complexes, we obtain
\beqn
\dots (\Lg^{\otimes r-q}\otimes \Lb^{\otimes q})^{\oplus {\binom{r}{ q}}}\to\dots \to \Lg^{\otimes r}\to (\Lg/\Lb)^{\otimes r}\to0.
\eeqn
Passing to the associated spectral sequence we obtain
\beq
\label{spectral sequence}
E_1^{p,-q} =  (\Lg^{\otimes r-q})^{\oplus {\binom{r}{ q}}}\otimes\oh^p(X,\Lb^{\otimes q}) \ \ \Longrightarrow
\oh^{p-q}(X, (\Lg/\Lb)^{\otimes r})\,.
\eeq
From this spectral sequence we conclude that the condition
\beq
\label{condition}
\oh^i(X, \Lb^{\otimes r})=0\text{ for all }\ i \geq r.
\eeq 
implies
\beqn
\Lg^{\otimes r} \to \oh^0(X,(\Lg/\Lb)^{\otimes r}) \qquad\text{is a surjection}\,.
\eeqn
Arguing as in Lemma~\ref{lem-sts} by decomposing tensors we conclude
\beqn
\text{if condition~\eqref{condition} holds for all $r$ then  $\Sym(\Lg^{\oplus r}) \to \oh^0(X,\Sym((\Lg/\Lb)^{\oplus r}))$ is onto for all $r$}.
\eeqn
Combining this with the previous discussion we obtain
\begin{thm} 
\label{criterion}
If condition~\eqref{condition} is satisfied for all $r$,  then $\cA_r$ has rational singularities for all values $r$. 
\end{thm}

\section{Cohomology of the vector bundles $\Lb^{\otimes r}$ for type $A_n$}

In the previous section we reduced the question of normality and the question of rational singularities of $\cA_r$ to the study of the cohomology of $\Lb^{\otimes r}$. 
We will now calculate the cohomology of these vector bundles  for small $r$ in type $A_{n-1}$, i.e., for $\Lg=\mathfrak{sl}_n$. We will go sufficiently far to obtain our counterexamples, but will not make an attempt for a complete general answer.

\subsection{The case $r=1$ and $\Lg$ of any type}\label{1}

This case is of course well known, easy and applies to any $\Lg$, but we include the details in any case. 

We will show that:
\beq\label{eqn-van1}
\oh^k(X,\Lb)=0\text{ for all }k.
\eeq
Consider the short exact sequence
\beqn
0\to\Ln\to\Lb\to\Lb/\Ln\to 0.
\eeqn
Since $\Lb/\Ln$ is a trivial vector bundle, $\oh^k(X,\Lb/\Ln)=0$ for all $k\geq 1$. As by~\eqref{Hodge} we have $\oh^k(X,\Ln)=0$ for $k\neq 1$, it follows that
\beqn
\oh^k(X,\Lb)=0\text{ for }k\geq 2.
\eeqn
From the short exact sequence
$$0\to\Lb\to\Lg\to\Lg/\Lb\to 0$$
we obtain the exact sequence
$$
0\to \oh^0(X,\Lb)\to \oh^0(X,\Lg)\to \oh^0(X,\Lg/\Lb)\to \oh^1(X,\Lb)\to 0.
$$
Moreover, $\oh^0(X,\Lg)\cong\Lg$ and the map $\Lg \to\oh^0(X,\Lg/\Lb)$ is an isomorphism as can be seen by a direct verification, for example. Thus we obtain~\eqref{eqn-van1}.

\subsection{The case $r=2$ and $\Lg$ of type $A_{n-1}$}\label{2}
 Let $\Lg=\mathfrak{sl}_n$, $n\geq 3$. We will show that
\beq\label{eqn-t1}
\oh^k(X,\Lb\ot\Lb)=\left\{\begin{array}{ll}0&\text{ if } k\neq 1\\\tk&\text{ if }k=1.\end{array}\right.
\eeq
and
\beq\label{eqn-tt3}
\oh^k(X,(\Lg/\Lb)^{\ot2})\cong \left\{\begin{array}{ll}(\Lg\ot\Lg)/\tk&\text{ if }k=0\\0&\text{ if } k\geq 1.\end{array}\right.
\eeq

First, we claim that 
\beq\label{eqn-tt1}
\text{the potential support of }\oh^k(X,\Lb\ot\Lb)\text{ is }\{0\}\text{ for all }k\geq 0.
\eeq

Let $\chi\in\chi(\Lb\ot\Lb)$. We will show that
$$
 \text{either }\oh^k(X,\cL_\chi)=0\text{ for all }k\geq 0\text{ or }\chi=w\cdot0\text{ for some }w\in W. 
 $$
 Recall that the bundle associate to $\Lh=\Lb/\Ln$ is trivial and thus, by making use of Lemma~\ref{distinct roots}, it suffices to prove the result for  $\chi\in\chi(\Ln\ot\Ln)$. The case when $\chi$ is a sum of distinct negative roots follows from~\eqref{BWB} and Lemma~\ref{distinct roots}. So it remains to consider the case when $\chi\in\chi(\Ln\ot\Ln)$ and  $\chi$ is not a sum of distinct negative roots. In that case $\chi=-2\alpha$ for a positive root $\alpha\in \Phi^+$ and $\text{ht}(\alpha)=1$ (if $\text{ht}(\alpha)>1$ then $\chi$ is a sum of distinct negative roots). Let us denote the set of simple roots with respect to $\Phi^+$ by $\Delta:=\{\alpha_i,i=1,\ldots,n-1\}$.
  Let $\alpha=\alpha_i$, $1\leq i\leq n-1$. If $i\leq n-2$, then $\la-2\alpha_i+\rho,(\alpha_i+\alpha_{i+1})^\vee\ra=0$; if $i\geq2$, then $\la-2\alpha_i+\rho,(\alpha_{i-1}+\alpha_{i})^\vee\ra=0$; thus $\oh^k(X,\cL_\chi)=0$ for all $k$. This finishes the proof of \eqref{eqn-tt1}.

Using the short exact sequence
\beqn
0\to \Lb\otimes\Lb\to\Lg\ot\Lb\to\Lg/\Lb\ot\Lb\to 0.
\eeqn
and 
$\oh^k(X,\Lg\ot\Lb)\cong\Lg\ot \oh^k(X,\Lb)=0$ for all $k$ (we make use of \eqref{eqn-van1}), we see that
\beqn
\oh^0(X,\Lb\ot\Lb)=0\text{ and }\oh^k(X,\Lb\ot\Lb)\cong \oh^{k-1}(X,\Lb\ot\Lg/\Lb)\text{ as }G\text{-modules for } k\geq 1.
\eeqn
Using the short exact sequence
\beqn
0\to \Lb\otimes\Lg/\Lb\to\Lg\ot\Lg/\Lb\to(\Lg/\Lb)^{\ot 2}\to 0
\eeqn
and $\oh^k(X,\Lg\otimes\Lg/\Lb)\cong\Lg\otimes \oh^k(X,\Lg/\Lb)=0$ for $k\geq 1$, we see that
\beqn
\oh^{k}(X,\Lb\ot\Lg/\Lb)\cong \oh^{k-1}(X,(\Lg/\Lb)^{\ot2}) \text{ as }G\text{-modules for }k\geq 2
\eeqn
and that we have an exact sequence of $G$-modules
\beqn
0\to \oh^0(X,\Lb\ot\Lg/\Lb)\to\Lg\ot\Lg\to \oh^0(X,(\Lg/\Lb)^{\ot2})\to \oh^1(X,\Lb\ot\Lg/\Lb)\to 0.
\eeqn
It follows that we have an isomorphism of $G$-modules
\begin{equation}\label{eqn-l3}
\oh^k(X,\Lb\ot\Lb)\cong \oh^{k-2}(X,(\Lg/\Lb)^{\ot2})\text{ for } k\geq 3
\end{equation}
and an exact sequence of $G$-modules
\beq\label{eqn-ll1}
0\to \oh^1(X,\Lb\otimes\Lb)\to\Lg\ot \Lg\to \oh^0(X,(\Lg/\Lb)^{\ot 2})\to \oh^2(X,\Lb\ot\Lb)\to 0.
\eeq

As $0$ is clearly not in the potential support of $\oh^k(X,(\Lg/\Lb)^{\ot2})$, it follows from \eqref{eqn-tt1}, \eqref{eqn-l3} and \eqref{eqn-ll1} that $\oh^k(X,\Lb\ot\Lb)=0$ for $k\geq 2$, and
\beqn
\oh^1(X,\Lb\ot\Lb)\cong\Hom_G(\tk,\Lg\ot\Lg)\cong\Hom_G(\Lg,\Lg)\cong\tk.
\eeqn
This completes the proof of \eqref{eqn-t1}. It also follows that
\beq\label{eqn-tt2}
\oh^k(X,\Lb\ot\Lg/\Lb)\cong\left\{\begin{array}{ll}\tk&\text{ if }k=0\\0&\text{ if } k\geq 1\end{array}\right.
\eeq
Finally, we conclude~\eqref{eqn-tt3} from \eqref{eqn-t1},  \eqref{eqn-l3}, and \eqref{eqn-ll1}.

\subsection{The case $r=3$ and $\Lg$ of type $A_{n-1}$} \label{3}
Let $\Lg=\mathfrak{sl}_n$, $n\geq 3$. We will show that
\begin{eqnarray}\label{eqn-tb1}
&&\oh^k(X,\Lb^{\ot 3})=0,\ k\leq 1\text{ or }{k\geq 4};\nonumber\\
&&
\oh^2(X,\Lb^{\ot 3})=\left\{\begin{array}{ll}\tk^{\oplus 2}\oplus L(\alpha_1+\alpha_2)^{\oplus 5}\oplus L(2\alpha_1+\alpha_2)\oplus L(\alpha_1+2\alpha_2)&\text{ if }n=3\\\tk^{\oplus 2}&\text{ if }n\geq 4;\end{array}\right.
\\
&&
\oh^3(X,\Lb^{\ot 3})=\left\{\begin{array}{ll}L(\alpha_1+2\alpha_2+\alpha_3)&\text{ if }n=4\\0&\text{ if }n\neq 4.\end{array}\right.\nonumber
\end{eqnarray}

We first show that
\beq
\label{eqn-pp3}
\begin{aligned}
 &\text{if }  n\geq 5   &&\PSupp(\oh^k(X,\Lb^{\ot 3}))=\{0\}\text{ for all }k, 
\\
&  \text{if $n=4$}  &&\PSupp(\oh^k(X,\Lb^{\ot 3}))=\left\{\begin{array}{ll}
\{0,\alpha_1+2\alpha_2+\alpha_3\}&\text{ if }k=3\\\{0\}&\text{ if }k\neq3,\end{array}\right.
\\
 &  \text{if $n=3$}  &&\PSupp(\oh^k(X,\Lb^{\ot3}))=\left\{\begin{array}{ll}
\{0,\alpha_1+\alpha_2,2\alpha_1+\alpha_2,\alpha_1+2\alpha_2\}&\text{ if }k=2\\
\{0,\alpha_1+\alpha_2\}&\text{ if }k=3\\\{0\}&\text{ if }k\neq 2,3.\end{array}\right.
\end{aligned}
\eeq

Let $\chi\in\chi(\Lb^{\ot 3})$. Recall that the bundle associate to $\Lh=\Lb/\Ln$ is trivial and thus, by making use of~\eqref{eqn-tt1}, it suffices to prove the result for  $\chi\in\chi(\Ln^{\ot 3})$. If $\chi\in\chi(\wedge^3\Ln)$, we see by \eqref{BWB} and Lemma~\ref{distinct roots} that either $H^k(X,\cL_\chi)=0$ for all $k$ or $\chi=w\cdot0$ for some $w\in W$. It remains to consider $\chi$ of the form $-2\alpha-\beta$, $\alpha,\beta\in \Phi^+$.

Let us introduce some notation. For $w\in W$, denote 
\beqn
\Phi_w:=\{\gamma\in \Phi^+\,|\,w^{-1}\gamma\in \Phi^-\}.
\eeqn
Suppose that $w=s_{i_1}s_{i_2}\cdots s_{i_l}$, where $s_{i}=s_{\alpha_i}$, is a reduced expression. Recall that we have
\beqn
\Phi_w=\{\alpha_{i_1},\ s_{i_1}(\alpha_{i_2}),\  s_{i_1}s_{i_2}(\alpha_{i_3}),\ \cdots,\ s_{i_1}\cdots s_{i_{l-1}}(\alpha_{i_l})\}.
\eeqn
In particular, the cardinality of $\Phi_w$ is the length $\ell(w)$.

For $\lambda\in X(T)$, we have
\beq
\label{dot}
w\cdot\lambda=w(\lambda)-\sum_{\gamma\in\Phi_w}\gamma.
\eeq

Suppose first that $\alpha\neq \beta$. We show that if there exists $w\in W$ such that $w\cdot(-2\alpha-\beta)\in X^+(T)$, then $w\cdot(-2\alpha-\beta)=0$ unless $n=3$ and $\chi=-3\alpha_1-\alpha_2$ or $-\alpha_1-3\alpha_2$.

As the dominant Weyl chamber is contained in the positive root cone, i.e., the inverse of the Cartan matrix has positive entries, we have
\beq
\begin{gathered}
\label{Cartan}
\text{Let $\lambda \in \bZ\Phi$ be an element in the root lattice which is}
\\
\text{dominant and not zero then $\lambda\in \bZ_{> 0}\al_1 + \dots +\bZ_{>0}\al_{n-1}$}\,.
\end{gathered} 
\eeq

Assume that $w\cdot(-2\alpha-\beta)$ is dominant and not equal to zero.  Making use of~\eqref{dot} and~\eqref{Cartan} we see that 
\beqn
w\cdot(-2\alpha-\beta) = -2w\al - w\beta - \sum_{\gamma\in\Phi_w}\gamma \in \bZ_{> 0}\al_1 + \dots +\bZ_{>0}\al_{n-1}\,.
\eeqn
Let us write $\alpha_0=\al_1+\dots \al_{n-1}$ for the highest root and then we can rephrase the above equality as 
\beqn
-2w\al - w\beta - \sum_{\gamma\in\Phi_w}\gamma \geq \al_0\,.
\eeqn
Clearly at least one of $w\al$ or $w\beta$ has to be negative. Note that if $w\beta$ is negative then $-w\beta\in\Phi_w$ and similarly for $w\al$.  Thus, if $w\al$ is negative and $w\beta$ is not we get 
\beqn
-w\al - w\beta  - \sum_{\gamma\in\Phi_w-\{-w\al\}}\gamma \geq \al_0\,.
\eeqn
But this is impossible as $-w\al \leq\al_0$. Similarly we see that it is impossible for  $w\al$ to be positive and for $w\beta$ to be negative. Hence, both $w\al$ and $w\beta$ have to be negative. In this case we see that
\beqn
-w\al  - \sum_{\gamma\in\Phi_w-\{-w\al,-w\beta\}}\gamma \geq \al_0\,.
\eeqn
This is only possble if  $w\alpha=-\alpha_0$, $\Phi_w=\{\al_0,-w\beta\}$, and then also $\ell(w)=2$. One sees directly that $w\al_0$ can be negative for $\ell(w)=2$ only when $n=3$. In that case the only possibilities are:
\beqn
-2\alpha_1-(\alpha_1+\alpha_2)=(s_1s_2)\cdot(\alpha_1+\alpha_2),\quad-2\alpha_2-(\alpha_1+\alpha_2)=(s_2s_1)\cdot(\alpha_1+\alpha_2)\,.
\eeqn
Thus we conclude that in this case the only possibility for potential support, in addition to $0$,   is  $\alpha_1+\alpha_2\in \PSupp\oh^2(X,\Lb^{\ot3})$ when $n=3$.

Suppose now that $\alpha=\beta$ and thus $\chi=-3\alpha$. If $\text{ht}\,\alpha\geq 3$, then $\chi$ is a sum of distinct negative roots  and we see  by~\eqref{BWB} and Lemma~\ref{distinct roots} that either $H^k(X,\cL_\chi)=0$ for all $k$ or $\chi=w\cdot0$ for some $w\in W$. Let us  assume next  that $\text{ht}\,\alpha= 2$ and we write $\alpha=\alpha_i+\alpha_{i+1}$, $1\leq i\leq n-2$. If $i\leq n-3$, then $\la-3\alpha+\rho,(\alpha_i+\alpha_{i+1}+\alpha_{i+2})^\vee\ra=0$ and if $i\geq 2$, then $\la-3\alpha+\rho,(\alpha_{i-1}+\alpha_{i}+\alpha_{i+1})^\vee\ra=0$. Thus in these cases  $\oh^k(X,\cL_\chi)=0$ for all $k$ and they do not contribute to the potential support. Therefore we are left to consider the case when $n=3$ and $i=1$. Then
\beqn
-3(\alpha_1+\alpha_2)=(s_1s_2s_1)\cdot(\alpha_1+\alpha_2).
\eeqn
Hence $\alpha_1+\alpha_2\in \PSupp\oh^3(X,\Lb^{\ot3})$ when $n=3$. 

Finally, let us  assume that $\text{ht}\,\alpha= 1$ and so  $\alpha=\alpha_i$, $1\leq i\leq n-1$. If $i\leq n-3$, we have $\la-3\alpha+\rho,(\alpha_i+\alpha_{i+1}+\alpha_{i+2})^\vee\ra=0$; if $i\geq 3$, we have $\la-3\alpha+\rho,(\alpha_{i-2}+\alpha_{i-1}+\alpha_{i})^\vee\ra=0$. This shows that $\oh^k(X,\cL_\chi)=0$ for all $k$. Thus we get no contribution to the potential support in these cases. This leaves us with the possibilities of  $n=3$ with $i=1,2$ and $n=4$ with $i=2$. In these cases, we have
\beqn
-3\alpha_1=(s_1s_2)\cdot(2\alpha_1+\alpha_2),\ \ -3\alpha_2=(s_2s_1)\cdot(\alpha_1+2\alpha_2) \ \ (n=3)
\eeqn
\beqn
-3\alpha_2=(s_2s_1s_3)\cdot(\alpha_1+2\alpha_2+\alpha_3)\ \ (n=4).
\eeqn 
Thus we obtain $2\alpha_1+\alpha_2, \alpha_1+2\alpha_2\in \PSupp\oh^2(X,\Lb^{\ot3})$ when $n=3$ and $\alpha_1+2\alpha_2+\alpha_3\in \PSupp\oh^3(X,\Lb^{\ot3})$ when $n=4$.

This completes the proof of \eqref{eqn-pp3} and we now turn to the proof of~\eqref{eqn-tb1}. 

Using the short exact sequence
\beqn
0\to\Lb^{\ot3}\to\Lg\ot\Lb^{\ot2}\to\Lb^{\ot2}\ot\Lg/\Lb\to 0
\eeqn
and \eqref{eqn-t1} we see that we have
\beqn
\oh^0(X,\Lb^{\ot 3})=0,\ \oh^k(X,\Lb^{\ot 3})\cong \oh^{k-1}(X,\Lb^{\ot2}\ot\Lg/\Lb)\text{ for } k\geq 3,
\eeqn
and also an exact sequence
\beq\label{eqn-pp4}
0\to \oh^0(X,\Lb^{\ot2}\ot\Lg/\Lb)\to \oh^1(X,\Lb^{\ot3})\to\Lg\to \oh^1(X,\Lb^{\ot2}\ot\Lg/\Lb)\to \oh^2(X,\Lb^{\ot 3})\to 0.
\eeq
Using the short exact sequence
\beqn
0\to\Lb^{\ot2}\ot\Lg/\Lb\to\Lg\ot\Lb\ot\Lg/\Lb\to\Lb\ot(\Lg/\Lb)^{\ot2}\to 0
\eeqn
and \eqref{eqn-tt2} we obtain
\beqn
\oh^k(X,\Lb^{\ot 2}\ot\Lg/\Lb)\cong \oh^{k-1}(X,\Lb\ot(\Lg/\Lb)^{\ot2})\text{ for}\ k\geq 2
\eeqn
and an exact sequence
\beq\label{eqn-pp5}
0\to \oh^0(X,\Lb^{\ot2}\ot\Lg/\Lb)\to \Lg\to \oh^0(X,\Lb\ot(\Lg/\Lb)^{\ot2})\to \oh^1(X,\Lb^{\ot 2}\ot\Lg/\Lb)\to 0.
\eeq
Using the short exact sequence
\beqn
0\to\Lb\ot(\Lg/\Lb)^{\ot2}\to\Lg\ot(\Lg/\Lb)^{\ot2}\to(\Lg/\Lb)^{\ot 3}\to 0
\eeqn
and \eqref{eqn-tt3} we obtain
\beqn
\oh^k(X,\Lb\ot(\Lg/\Lb)^{\ot2})\cong \oh^{k-1}(X,(\Lg/\Lb)^{\ot 3})\text{ for } k\geq 2
\eeqn
and an exact sequence
\beqn
0\to \oh^0(X,\Lb\ot(\Lg/\Lb)^{\ot2})\to \Lg\ot \oh^0(X,(\Lg/\Lb)^{\ot2})\to \oh^0(X,\Lg/\Lb^{\ot3})\to \oh^1(X,\Lb\ot(\Lg/\Lb)^{\ot2})\to 0.
\eeqn

It follows that we have an isomorphism of $G$-modules
\beqn
\oh^k(X,\Lb^{\ot 3})\cong \oh^{k-3}(X,(\Lg/\Lb)^{\ot 3})\text{ for } k\geq 4,
\eeqn
and an exact sequence of $G$-modules
\beq\label{eqn-pp2}
0\to \oh^0(X,\Lb\ot(\Lg/\Lb)^{\ot2})\to \Lg\ot \oh^0(X,(\Lg/\Lb)^{\ot2})\to \oh^0(X,(\Lg/\Lb)^{\ot3})\to \oh^3(X,\Lb^{\ot3})\to 0.
\eeq
Thus, as  $0$ is not in the potential support of $\oh^k(X,(\Lg/\Lb)^{\ot3})$, we conclude that
\beq
\label{h3}
 \oh^k(X,\Lb^{\ot3})\ \ \  \text{does not contain the trivial representation for $k\geq 3$}\,.
\eeq
Thus, we conclude from~\eqref{eqn-pp3} that $\oh^k(X,\Lb^{\ot3})=0$ for $k\geq 4$.

Also, it follows from  \eqref{eqn-pp3}, \eqref{eqn-pp4}, and \eqref{eqn-pp5} that
\beq\label{eqn-van2}
\oh^0(X,\Lb^{\ot2}\ot\Lg/\Lb)\cong \oh^1(X,\Lb^{\ot3})=0.
\eeq
Thus, we have shown, in particular, that $\oh^k(X,\Lb^{\ot3})=0$ for $k=0,1$. Thus we have obtained the first claim of~\eqref{eqn-tb1}. Before proceeding further, we record one more general fact which we obtain from \eqref{eqn-van2}, \eqref{eqn-pp4}, and \eqref{eqn-pp5}:
\begin{subequations}
\label{auxiliary}
\beq
\label{auxiliarya}
 \oh^{0}(X,\Lb\ot(\Lg/\Lb)^{\ot 2}) \cong \oh^1(X,\Lb^{\ot 2}\ot\Lg/\Lb)\oplus\Lg\cong\oh^2(X,\Lb^{\ot 3})\oplus\Lg^{\oplus 2}
\eeq
\beq
\oh^{1}(X,\Lb\ot(\Lg/\Lb)^{\ot 2})\cong\oh^2(X,\Lb^{\ot 2}\ot\Lg/\Lb)\cong  \oh^3(X,\Lb^{\ot 3})
\eeq
\beq
\oh^{k-1}(X,\Lb\ot(\Lg/\Lb)^{\ot 2})\cong \oh^k(X,\Lb^{\ot 2}\ot\Lg/\Lb)= 0  \text{ if }k\neq 1,2.
\eeq
\end{subequations}

We now argue with specific values of $n$.

Assume that $n\geq 5$.  It follows from~\eqref{h3}
and \eqref{eqn-pp3} that 
\beqn
\oh^3(X,\Lb^{\ot 3})=0 \ \ \text{and}\ \ \oh^2(X,\Lb^{\ot 3})\cong\tk^{\oplus c}\,,
\eeqn
for some $c$  which we determine later. 

Assume that $n=4$. It follows from \eqref{h3} and \eqref{eqn-pp3}  that 
\beqn
\oh^3(X,\Lb^{\ot 3})\cong L(\alpha_1+2\alpha_2+\alpha_3)  \ \ \text{and}\ \  \oh^2(X,\Lb^{\ot 3})\cong\tk^{\oplus a}\,,
\eeqn
for some $a$ which we determine later. 

Assume that $n=3$. As $\alpha_1+\alpha_2$ is not in the potential support of $\oh^0(X,(\Lg/\Lb)^{\ot3})$ we conclude from~\eqref{eqn-pp3}, \eqref{eqn-pp2} and~\eqref{h3} that 
\beqn
\oh^3(X,\Lb^{\ot3})=0.
\eeqn
Thus $\oh^2(X,\Lb^{\ot 3})$ is the only non-vanishing cohomology group in this case. Making use of Borel-Weil-Bott \eqref{BWB} just as in our argument for \eqref{eqn-pp3} we see that 
\beqn
\oh^*(X,\gr(\Lb^{\ot 3}))\cong \tk^{\oplus b}+ L(\alpha_1+\alpha_2)^{\oplus 5}+L(2\alpha_1+\alpha_2)+L(\alpha_1+2\alpha_2)\,.
\eeqn
for some $b$. 
Therefore we get:
\beqn
\oh^2(X,\Lb^{\ot 3})\cong\oh^*(X,\gr(\Lb^{\ot 3}))\cong \tk^{\oplus b}\oplus L(\alpha_1+\alpha_2)^{\oplus 5}\oplus L(2\alpha_1+\alpha_2)\oplus L(\alpha_1+2\alpha_2).
\eeqn
We now determine $a$, $b$, and $c$. Now, 
\beqn
\tk^{\oplus a}\cong\tk^{\oplus b}\cong\tk^{\oplus c}\cong\Hom_G(\tk,\oh^2(X,\Lb^{\ot 3}))\,.
\eeqn
By~\eqref{auxiliarya} we get 
\beqn
\begin{gathered}
\Hom_G(\tk,\oh^2(X,\Lb^{\ot 3})) \cong \Hom_G(\tk,\oh^2(X,\Lb^{\ot 3})\oplus\Lg^{\oplus 2})\cong  \Hom_G(\tk,\oh^{0}(X,\Lb\ot(\Lg/\Lb)^{\ot 2}))\cong
\\
\cong \Hom_G(\tk,\Lg\ot \oh^0(X,(\Lg/\Lb)^{\ot2}))\cong (\Lg\ot\Lg\ot\Lg)^G\cong\tk^{\oplus 2}\,;
\end{gathered}
\eeqn
where in the third equality we have made use of~\eqref{eqn-pp2} and the fact that the trivial representation does not occur in $\oh^0(X,(\Lg/\Lb)^{\ot3})$,  in the fourth equality we made use of~\eqref{eqn-tt3}. The last equality is classical and can also be verified by a direct calculation: the two invariant tensors are $(x,y,z)\mapsto \operatorname{Tr}(xyz)$ and $(x,y,z)\mapsto \operatorname{Tr}(yxz)$.
This completes the proof of \eqref{eqn-tb1}.

\subsection{The case $r=4$ and $\Lg$ of type $A_{n-1}$}Let $\Lg=\mathfrak{sl}_n$, $n\geq 6$. 

In the previous cases we obtained complete information of the cohomology for all values of $n\geq3$. For $r=4$ we will not make an attempt to get a complete answer, but will just obtain enough information for our counterexample. In particular, we already have enough information to prove the Cohen-Macaulay property for $\mathfrak{sl}_3$. In the cases $n=4,5$ the answer is probably obtainable with our techniques but is more complicated. 

We will show that if $n\geq 7$, then
\begin{subequations}
\label{eqn-sub}
\beq
\oh^k(X,\Lb^{\ot 4})=
0\text{ if }k\neq 2,3
\eeq
and if $n=6$, then
\beq
\oh^k(X,\Lb^{\ot 4})=\left\{\begin{array}{ll}
L(\alpha_1+2\alpha_2+3\alpha_3+2\alpha_4+\alpha_5)&\text{ if }k=5\\0&\text{ if }k\neq 2,3,5\,. \end{array}\right.
\eeq
\end{subequations}
\begin{rmk}
The $\oh^2(X,\Lb^{\ot 4})$ and  $\oh^3(X,\Lb^{\ot 4})$ both consist of a number of copies of the trivial representation. 
\end{rmk}

We first show that
\begin{subequations}\label{eqn-pt1}
\beq
\text{ if }n\geq 7\ \ \  \PSupp\oh^k(X,\Lb^{\ot 4})=\{0\}\text{ for all }k\geq 0; 
\eeq
and
\beq
\text{ if } n=6 \ \  \PSupp\oh^k(X,\Lb^{\ot 4}) =\left\{\begin{array}{ll}\{0\}&\text{ if }k\neq 5\\\{0,\alpha_1+2\alpha_2+3\alpha_3+2\alpha_4+\alpha_5\}&\text{ if }k=5\,.\end{array}\right.
\eeq
\end{subequations}

 Let $\chi\in\chi(\Lb^{\ot 4})$. Recall that the bundle associate to $\Lh=\Lb/\Ln$ is trivial. Thus, by making use of~\eqref{eqn-pp3}, it suffices to prove the result for $\chi\in\chi(\Ln^{\ot 4})$. If $\chi\in\chi(\wedge^4\Ln)$, we see by~\eqref{BWB} and Lemma~\ref{distinct roots} that either $H^k(X,\cL_\chi)=0$ for all $k$ or $\chi=w\cdot0$ for some $w\in W$. Thus we are reduced to considering the case when $\chi$ is not a sum of distinct roots. 

Assume first that $\chi=-2\alpha-\beta-\gamma$, where $\alpha,\beta,\gamma\in \Phi^+$ are distinct. We will make use of the notation and argue in the similar manner as in the previous subsection~\ref{3} . Suppose that there exists a $w\in W$, such that  $w\cdot\chi\in X^+(T)$ and $w\cdot\chi\neq 0$. Making use of~\eqref{dot} and~\eqref{Cartan} we conclude, as in subsection~\ref{3}, that  
\beqn
w\cdot(-2\alpha-\beta-\gamma) = w(-2\alpha-\beta-\gamma)  - \sum_{\gamma\in\Phi_w}\gamma \geq \al_0\,.
\eeqn
Further arguing as in subsection~\ref{3} we conclude that $w$ has to satisfy that $w\alpha=-\alpha_0,w\beta<0,w\gamma<0$ and $\ell(w)=3$. But,  $\alpha_0\in \Phi_w$ only if $l(w)\geq n-1$ and we have assumed that $n\geq 6$. 

Assume next that $\chi=-2\alpha-2\beta$, where $\alpha,\beta\in \Phi^+$ are distinct. Suppose that there exists a $w\in W$, such that  $w\cdot\chi\in X^+(T)$ and $w\cdot\chi\neq 0$. Making use of~\eqref{Cartan} again we see  that 
\beqn
w(-2\alpha-2\beta+\rho)-\rho= \al_0+\lambda_0\,
\eeqn
for some $\lambda_0\in\bZ_+\Phi^+$.
In particular, we have
\begin{eqnarray*}
&&\la-2\alpha-2\beta+\rho,-2\alpha-2\beta+\rho\ra=\la w(-2\alpha-2\beta+\rho),w(-2\alpha-2\beta+\rho)\ra
\\
&&=\la\alpha_0+\lambda_0+\rho,\alpha_0+\lambda_0+\rho\ra\geq\la\alpha_0+\rho,\alpha_0+\rho\ra=2n+\la\rho,\rho\ra.
\end{eqnarray*}
It follows that
\beqn
\la\alpha,\alpha\ra+\la\beta,\beta\ra+2\la\alpha,\beta\ra-\la\alpha+\beta,\rho\ra\geq n/2.
\eeqn
This can only happen when 
\beqn
n=6,\ \la\alpha,\beta\ra=1, \la\alpha,\rho\ra=2, \ \la\beta,\rho\ra=1.
\eeqn
Here we can and assume that $\text{ht}\alpha=\la\alpha,\rho\ra\geq\text{ht}\beta=\la\beta,\rho\ra$. 

Suppose that $n=6$ and $\beta=\alpha_i$, $1\leq i\leq 5$. Then $\alpha=\alpha_i+\alpha_{i+1}$ or $\alpha=\alpha_{i-1}+\alpha_i$. In the first case we have that $-2\alpha-2\beta=-4\alpha_i-2\alpha_{i+1}$, and
\beqn
\left\{\begin{array}{ll}\la-4\alpha_i-2\alpha_{i+1}+\rho,(\alpha_{i-1}+\alpha_i)^\vee\ra=0&\text{ if }i\geq 2
\\
\la-4\alpha_i-2\alpha_{i+1}+\rho,(\sum_{j=i}^{i+3}\alpha_{j})^\vee\ra=0&\text{ if }i\leq 2.\end{array}\right.
\eeqn
In the second case we have that $-2\alpha-2\beta=-2\alpha_{i-1}-4\alpha_{i}$ and
\beqn
\left\{\begin{array}{ll}\la-2\alpha_{i-1}-4\alpha_{i}+\rho,(\alpha_{i}+\alpha_{i+1})^\vee\ra=0&\text{ if }i\leq 4
\\
\la-2\alpha_{i-1}-4\alpha_{i}+\rho,(\sum_{j=i-3}^{i}\alpha_{j})^\vee\ra=0&\text{ if }i\geq 4.\end{array}\right.
\eeqn
It follows that in these cases $-2\alpha-2\beta$ does not contribute to the potential support as $H^k(X,\cL_{-2\alpha-2\beta})=0$ for all $k$.

Assume that $\chi=-3\alpha-\beta$, where $\alpha,\beta\in \Phi^+$ are distinct roots. Suppose that there exists a $w\in W$, such that  $w\cdot\chi\in X^+(T)$ and $w\cdot\chi\neq 0$. Argue as above, we have
\beqn
9\la\alpha,\alpha\ra+\la\beta,\beta\ra+6\la\alpha,\beta\ra-6\la\alpha,\rho\ra-2\la\beta,\rho\ra\geq 2n.
\eeqn
This can happen only if $n\leq 8$, $\la\alpha,\beta\ra=1$, $3\la\alpha,\rho\ra+\la\beta,\rho\ra\leq 13-n$, or if $n=6$, $\la\alpha,\beta\ra=0$, and $\la\alpha,\rho\ra=\la\beta,\rho\ra=1$. More precisely, if $n=6$ and $\la\alpha,\beta\ra=1$, then the possible values for $(\la\alpha,\rho\ra,\la\beta,\rho\ra)$ are $(1,2),(1,3),(1,4),(2,1)$; if $n=7$ and $\la\alpha,\beta\ra=1$, then the possible values for $(\la\alpha,\rho\ra,\la\beta,\rho\ra)$ are $(1,2),(1,3)$; and if $n=8$ and $\la\alpha,\beta\ra=1$, then $(\la\alpha,\rho\ra,\la\beta,\rho\ra)=(1,2)$.

Suppose first that $\la\alpha,\beta\ra=1$ and $\alpha=\alpha_i$, $1\leq i\leq n-1$. Then $\beta=\sum_{k=i}^{j}\alpha_k$ for some $j\geq i+1$ or $\beta=\sum_{k=j}^{i}\alpha_k$ for some $j\leq i-1$. In the first case we have $-3\alpha-\beta=-4\alpha_i-\sum_{k=i+1}^{j}\alpha_k$ and
\beqn
\left.\begin{array}{lll}\la-3\alpha-\beta+\rho,\alpha_j^\vee\ra=0&\text{ if }j\neq i+1\\
\la-3\alpha-\beta+\rho,(\sum_{k=i}^{i+3}\alpha_k)^\vee\ra=0&\text{ if }j= i+1\text{ and }i\leq n-4\\
\la-3\alpha-\beta+\rho,(\sum_{k=i-2}^{i}\alpha_k)^\vee\ra=0&\text{ if }j= i+1\text{ and }i\geq 3;\end{array}\right.
\eeqn
in the second case we have $-3\alpha-\beta=-\sum_{k=j}^{i-1}\alpha_k-4\alpha_i$ and
\beqn
\left.\begin{array}{lll}\la-3\alpha-\beta+\rho,\alpha_j^\vee\ra=0&\text{ if }j\neq i-1\\
\la-3\alpha-\beta+\rho,(\sum_{k=i}^{i+2}\alpha_k)^\vee\ra=0&\text{ if }j= i-1\text{ and }i\leq n-3\\
\la-3\alpha-\beta+\rho,(\sum_{k=i-3}^{i}\alpha_k)^\vee\ra=0&\text{ if }j= i-1\text{ and }i\geq 4.\end{array}\right.
\eeqn
Thus in all these cases $-3\alpha-\beta$ dose not contribute to the potential support.

Suppose next that $n=6$, $\la\alpha,\beta\ra=1$, $\la\alpha,\rho\ra=2$ and $\la\beta,\rho\ra=1$. Let us write $\beta=\alpha_i$. Then $\alpha=\alpha_i+\alpha_{i+1}$ or $\alpha=\alpha_{i-1}+\alpha_i$. We have
\beqn
\left.\begin{array}{ll}\la-4\alpha_i-3\alpha_{i+1}+\rho,(\sum_{k=i-1}^{i+1}\alpha_k)^\vee\ra=0&\text{ if }i\geq 2\\\la-4\alpha_i-3\alpha_{i+1}+\rho,(\sum_{k=i}^{i+3}\alpha_k)^\vee\ra=0&\text{ if }i\leq 2\end{array}\right.
\eeqn
or
\beqn
\left.\begin{array}{ll}\la-3\alpha_{i-1}-4\alpha_{i}+\rho,(\sum_{k=i-1}^{i+1}\alpha_k)^\vee\ra=0&\text{ if }i\leq 4\\\la-3\alpha_{i-1}-4\alpha_{i}+\rho,(\sum_{k=i-3}^{i}\alpha_k)^\vee\ra=0&\text{ if }i\geq 4.\end{array}\right.
\eeqn
Thus again in this case $-3\alpha-\beta$ dose not contribute to the potential support.

Suppose now that $n=6$, $\la\alpha,\beta\ra=0$, $\la\alpha,\rho\ra=\la\beta,\rho\ra=1$. Let us write $\alpha=\alpha_i$ and $\beta=\alpha_j$. Then either $j\geq i+2$ or $j\leq i-2$. We have
\beqn
\left.
\begin{array}{ll}
\la-3\alpha_i-\alpha_j+\rho,(\sum_{k=i}^{i+2}\alpha_k)^\vee\ra=0&\text{ if }j>i+3\\
\la-3\alpha_i-\alpha_j+\rho,(\sum_{k=i}^{i+1}\alpha_k)^\vee\ra=0&\text{ if }j=i+2\\
\la-3\alpha_i-\alpha_j+\rho,(\sum_{k=i}^{i+3}\alpha_k)^\vee\ra=0&\text{ if }j=i+3\\
\la-3\alpha_i-\alpha_j+\rho,(\sum_{k=i-2}^{i}\alpha_k)^\vee\ra=0&\text{ if }j<i-3\\
\la-3\alpha_i-\alpha_j+\rho,(\sum_{k=i-1}^{i}\alpha_k)^\vee\ra=0&\text{ if }j=i-2\\
\la-3\alpha_i-\alpha_j+\rho,(\sum_{k=i-3}^{i}\alpha_k)^\vee\ra=0&\text{ if }j=i-3.
\end{array}
\right.
\eeqn
Thus also in this case $-3\alpha-\beta$ dose not contribute to the potential support.

Finally, assume that $\chi=-4\alpha$. If $\het(\alpha)\geq 4$, then $\chi$ is a sum of distinct negative roots, so we can assume that $\het(\alpha)\leq 3$. Suppose that $\alpha=\alpha_i+\alpha_{i+1}+\alpha_{i+2}$. If $i\leq n-4$, then $\la-4\alpha+\rho,(\sum_{j=i}^{i+3}\alpha_j)^\vee\ra=0$; if $i\geq 2$, then $\la-4\alpha+\rho,(\sum_{j=i-1}^{i+2}\alpha_j)^\vee\ra=0$; thus $\oh^k(X,\cL_\chi)=0$ for all $k$.

Suppose that $\alpha=\alpha_i+\alpha_{i+1}$. If $i\leq n-4$, then $\la-4\alpha+\rho,(\sum_{j=i}^{i+3}\alpha_j)^\vee\ra=0$; if $i\geq 3$, then $\la-4\alpha+\rho,(\sum_{j=i-2}^{i+1}\alpha_j)^\vee\ra=0$ and thus $\oh^k(X,\cL_\chi)=0$ for all $k$.

Suppose that $\alpha=\alpha_i$. If $i\leq n-4$, then $\la-4\alpha+\rho,(\sum_{j=i}^{i+3}\alpha_j)^\vee\ra=0$; if $i\geq 4$, then $\la-4\alpha+\rho,(\sum_{j=i-3}^{i}\alpha_j)^\vee\ra=0$; thus $\oh^k(X,\cL_\chi)=0$ for all $k$ unless when $n=6$ and $i=3$. In this case we have 
$$
-4\alpha_3=(s_5s_1s_4s_2s_3)\cdot(\alpha_1+2\alpha_2+3\alpha_3+2\alpha_4+\alpha_5)\,.
$$

This completes the proof of \eqref{eqn-pt1}.

Consider the short exact sequences
\beqn
0\to\Lb\ot\Lb^{\ot i}\ot(\Lg/\Lb)^{\ot(3-i)}\to\Lg\ot\Lb^{\ot i}\ot(\Lg/\Lb)^{\ot(3-i)}\to\Lg/\Lb\ot\Lb^{\ot i}\ot(\Lg/\Lb)^{\ot(3-i)}\to 0,\ i=0,1,2,3.
\eeqn
From the  exact sequence with $i=3$ we conclude, making use of~\eqref{eqn-tb1}, that
\beqn
\oh^0(X,\Lb^{\ot 4})=0,\ \oh^k(X,\Lb^{\ot 4})\cong \oh^{k-1}(X,\Lb^{\ot3}\ot\Lg/\Lb)\text{ for } k\geq 4\text{ or }k=1
\eeqn
and we further obtain the exact sequence
\beqn
0\to \oh^1(X,\Lb^{\ot3}\ot\Lg/\Lb)\to \oh^2(X,\Lb^{\ot 4})\to\Lg^{\oplus 2}\to \oh^2(X,\Lb^{\ot3}\ot\Lg/\Lb)\to \oh^3(X,\Lb^{\ot 4})\to 0\,.
\eeqn
Using the exact sequence with $i=2$ and making use of~\eqref{auxiliary} and~\eqref{eqn-tb1} we get
\beqn
\oh^0(X,\Lb^{\ot3}\ot\Lg/\Lb)=0,\ \oh^k(X,\Lb^{\ot3}\ot\Lg/\Lb)\cong \oh^{k-1}(X,\Lb^{\ot2}\ot(\Lg/\Lb)^{\ot2})\text{ for } k\geq 3
\eeqn
and we obtain the exact sequence
\begin{multline*}
0\to \oh^0(X,\Lb^{\ot2}\ot(\Lg/\Lb)^{\ot2})\to \oh^1(X,\Lb^{\ot 3}\ot\Lg/\Lb)\to\Lg^{\oplus 2}\oplus (\Lg\otimes\Lg)^{\oplus 2}\to
\\
\to \oh^1(X,\Lb^{\ot2}\ot(\Lg/\Lb)^{\ot2})\to \oh^2(X,\Lb^{\ot 3}\ot\Lg/\Lb)\to 0\,.
\end{multline*}
Using the exact sequence with $i=1$ and making use of~\eqref{auxiliary} and~\eqref{eqn-tb1} we get
\beqn
\oh^k(X,\Lb^{\ot2}\ot(\Lg/\Lb)^{\ot2})\cong \oh^{k-1}(X,\Lb\ot\Lg/\Lb^{\ot 3})\text{ for } k\geq 2,
\eeqn
and the exact sequence
\beqn
0\to \oh^0(X,\Lb^{\ot 2}\ot(\Lg/\Lb)^{\ot2})\to\Lg^{\oplus 2}\oplus (\Lg\otimes\Lg)\to \oh^0(X,\Lb\ot\Lg/\Lb^{\ot 3})\to \oh^1(X,\Lb^{\ot 2}\ot(\Lg/\Lb)^{\ot2})\to 0\,.
\eeqn
Using the exact sequence with $i=0$ and making use of~\eqref{auxiliary} and~\eqref{eqn-tb1} we get
\beqn
\oh^k(X,\Lb\ot(\Lg/\Lb)^{\ot3})\cong \oh^{k-1}(X,(\Lg/\Lb)^{\ot 4}),\ k\geq 2,
\eeqn
and we obtain the exact sequence
\beqn
0\to \oh^0(X,\Lb\ot(\Lg/\Lb)^{\ot3})\to\Lg\ot \oh^0(X,(\Lg/\Lb)^{\ot3})\to \oh^0(X,(\Lg/\Lb)^{\ot 4})\to \oh^1(X,\Lb\ot(\Lg/\Lb)^{\ot3})\to 0.
\eeqn
It follows that
\beq\label{eqn-i1}
\oh^k(X,\Lb^{\ot4})=0,\ k=0,1 \ \ \text{and} \ \ \ \oh^k(X,\Lb^{\ot4})\cong \oh^{k-4}(X,(\Lg/\Lb)^{\ot 4}),\ k\geq 5;
\eeq
and
\begin{multline}\label{eqn-i2}
0\to\oh^0(X,\Lb\ot(\Lg/\Lb)^{\ot3})\to\Lg\ot H^0(X,(\Lg/\Lb)^{\ot3})\to 
\\
\to \oh^0(X,(\Lg/\Lb)^{\ot4})\to H^4(X,\Lb^{\ot 4})\to 0.
\end{multline}

As $0$ is not in the potential support of $H^k(X,(\Lg/\Lb)^{\ot 4})$, using \eqref{eqn-pt1}, \eqref{eqn-i1} and \eqref{eqn-i2} we see that \eqref{eqn-sub} follows.

\subsection{Type $B_2$.}\label{ssec-B2}
 As the case $r=1$ applies to any type, we consider the case $r=2$ here. We write $\Phi^+=\{\alpha_1,\alpha_2,\alpha_1+\alpha_2,2\alpha_1+\alpha_2\}$, where $\la\alpha_1,\alpha_2^\vee\ra=-1$ and $\la\alpha_2,\alpha_1^\vee\ra=-2$.

Just as in subsection~\ref{2} one checks that:
$$
\PSupp(\oh^k(X,\Lb\ot\Lb)) =\left\{\begin{array}{ll}\{ 0\}&\text{ if }k=0,1,3\\\{0,\alpha_1+\alpha_2\}&\text{ if }k=2\\\emptyset&\text { if }k\geq 4.\end{array}\right.
$$
Proceeding as in~\ref{2} we see that:
\beq
\label{C2}
\oh^{k}(X,\Lb\ot\Lb)\ = \ \begin{cases}
0 &\text{if $k=0$ of $k\geq 3$}
\\
\tk &\text{if $k=1$}
\\
 L(\alpha_1+\alpha_2) &\text{if $k=2$}\,.
\end{cases}
\eeq
The only difference to the argument in subsection~\ref{2} is that the weight $\alpha_1+\alpha_2$ occurs for $k=2$ only and it occurs precisely once in that case; otherwise the argument is the same.

\section{Results and  counterexamples}

\subsection{The case of $A_2$} In this case we have the following:
\begin{thm}
For $\Lg=\mathfrak{sl}_3$ the variety $\cA_r$ has rational singularities, and thus Cohen-Macaulay, for all $r$. 
\end{thm}

\begin{proof}
According to Lemma~\ref{criterion} it suffices to show that
\beqn
\oh^i(X, \Lb^{\otimes r})=0\text{ for all } i\geq r\geq 1.
\eeqn
By dimension reasons we only have to consider the cases $r=1,2,3$. These cases are treated in subsections~\ref{1},\ \ref{2}, and\ \ref{3}, respectively. 

The fact that $\cA_r$ is Cohen-Macaulay follows; see for example \cite[Page 50]{K}.

\end{proof}

\begin{rmk}
By a slight modification of our methods one can show that this result holds for all characteristics above 3. 
\end{rmk}

\subsection{The case of $B_2$ and $r=2$} We will show that $\cA_2$ is not normal in the case of $B_2$. For $\cA_2$ to be normal, the map
\beqn
\Sym (\Lg\oplus \Lg) \to \oh^0(X,\Sym(\Lg/\Lb\oplus \Lg/\Lb))
\eeqn
has to be onto. In particular this has to hold for $\Sym^2$ and hence,  by decomposing the $\Sym^2$  on both sides, the map
\beqn
 \Lg\otimes \Lg \to \oh^0(X,\Lg/\Lb\otimes \Lg/\Lb)
\eeqn
has to be onto. Making use of the spectral sequence~\eqref{spectral sequence}, \eqref{eqn-van1}, and  \eqref{C2} we get that 
$$
E_1^{p,-q} =   (\Lg^{\otimes 2-q})^{\oplus{2\choose q}}\otimes\oh^p(X,\Lb^{\otimes q})\ = \ \begin{cases}
\Lg^{\otimes 2} & \text{if $p=q=0$}\\
 L(\alpha_1+\alpha_2) 	& \text{if $p=q=2$}\\
0	& \text{otherwise if $p-q\geq 0$\,.}		
\end{cases}
$$
As the term $E_1^{2,-2}$ must survive in the spectral sequence we see that $\Lg^{\otimes 2} \to \oh^0(X, (\Lg/\Lb)^{\otimes 2} )$ cannot be onto and so $\cA_2$ is not normal in this case. 

\begin{rmk}Using the same method as in the previous section, one can show that
\beqn
\PSupp(H^k(X,\wedge^3(\Lb^{\oplus 2})))=\left\{\begin{array}{cc}\{0\}&\text{ if }k=0,1,4\\
\{0,\alpha_1+\alpha_2,2\alpha_1+\alpha_2\}&\text{ if }k=2\\
\{0,\alpha_1+\alpha_2\}&\text{ if }k=3.
\end{array}\right.
\eeqn
In subsection \ref{ssec-B2} we have shown that $H^i(X,\wedge^2(\Lb^{\oplus 2}))=0$ for $i\geq 3$. Making use of the Koszul complex~\eqref{koszul complex}, we see that $H^k(X,\wedge^3(\Lb^{\oplus 2})))\cong H^{k-3}(X,\Sym^3((\Lg/\Lb)^{\oplus 2}))$. Since $0$ is not in the potential support of $H^{k}(X,\Sym^3((\Lg/\Lb)^{\oplus 2}))$, we conclude that
\beqn
H^4(X,\wedge^3(\Lb^{\oplus 2}))=H^{1}(X,\Sym^3((\Lg/\Lb)^{\oplus 2}))=0.
\eeqn
Now the dimension reasons imply that $H^i(X,\wedge^j(\Lb^{\oplus 2}))=0$ for all $j$ and all $i\geq j+1$, which in turn implies that
\beqn
H^{i}(X,\Sym((\Lg/\Lb)^{\oplus 2}))=0\text{ for all }i\geq 1.
\eeqn
Thus we see that the normalization of $\cA_2$ for $B_2$ has rational singularities.
\end{rmk}

\subsection{The case of $A_3$ and $r=3$}

Making use the spectral sequence~\eqref{spectral sequence}, \eqref{eqn-van1}, \eqref{eqn-t1}, and \eqref{eqn-tb1} we get that 
$$
E_1^{p,-q} =   (\Lg^{\otimes 3-q})^{\oplus{3\choose q}}\otimes\oh^p(X,\Lb^{\otimes q})\ = \ \begin{cases}
\Lg^{\otimes 3} & \text{if $p=q=0$}\\
L(\alpha_1+2\alpha_2+\alpha_3) 	& \text{if $p=q=3$}\\
0	& \text{otherwise if $p-q\geq 0$\,.}		
\end{cases}
$$
As the term $E_1^{3,-3}$ must survive in the spectral sequence we see that $\Lg^{\otimes 3} \to \oh^0(X, (\Lg/\Lb)^{\otimes 3} )$ cannot be onto and so, arguing as in the $B_2$ case above, we see that $\cA_3$ is not normal. 

\begin{rmk}Similar to the case of $\cA_2$ for $B_2$, the normalization of $\cA_3$ for type $A_3$ has rational singularities. One proceeds the same way as in that case. In particular, we have
\beqn
\PSupp(H^k(X,\wedge^4(\Lb^{\oplus 3})))=\left\{\begin{array}{cc}\{0\}&\text{ if }k=5,6\\
\{0,\alpha_1+\alpha_2+\alpha_3,\alpha_1+2\alpha_2+\alpha_3\}&\text{ if }k=4\\
\left\{\displaystyle{\substack{0,\,\alpha_1+\alpha_2+\alpha_3,\,\alpha_1+2\alpha_2+\alpha_3,\\\alpha_1+2\alpha_2+2\alpha_3,\,2\alpha_1+2\alpha_2+\alpha_3}}\right\}&\text{ if }k=3,
\end{array}\right.
\eeqn
\beqn
\PSupp(H^6(X,\wedge^5(\Lb^{\oplus 3})))=\{\alpha_1+\alpha_2+\alpha_3,0\}.
\eeqn
Note that $0$ and $\alpha_1+\alpha_2+\alpha_3$ are not in the potential support of $H^1(X,\Sym^5((\Lg/\Lb)^{\oplus 3}))$.
\end{rmk}

\subsection{The case of $A_5$ and $r=4$}

Making use the spectral sequence~\eqref{spectral sequence}, \eqref{eqn-van1}, \eqref{eqn-t1}, \eqref{eqn-tb1}, and \eqref{eqn-sub} we see  that 
$$
E_1^{p,-q} =   (\Lg^{\otimes 3-q})^{\oplus{3\choose q}}\otimes\oh^p(X,\Lb^{\otimes q})\ = \ \begin{cases}
L(\alpha_1+2\alpha_2+3\alpha_3+2\alpha_4+\alpha_5) 	& \text{if $p=5$ and $q=4$}\\
0	& \text{otherwise if $p-q>0 $.}		
\end{cases}
$$
As the term $E_1^{5,-4}$ must survive in the spectral sequence we see that $\oh^1(X, (\Lg/\Lb)^{\otimes 4})=L(\alpha_1+2\alpha_2+3\alpha_3+2\alpha_4+\alpha_5)$. In particular, it is not zero and hence  in this case the normalization of $\cA_4$ does not have rational singularities.

\end{document}